\documentclass{amsart}
\usepackage{amssymb, latexsym}

\DeclareMathOperator{\Tr}{Tr}

\DeclareMathOperator{\M}{\mathcal{M}}
\DeclareMathOperator{\K}{\mathcal{K}}
\DeclareMathOperator{\N}{\mathcal{N}}

\DeclareMathOperator{\rad}{rad}

\DeclareMathOperator{\Symm}{Symm}
\def\matrices{M_{m\times n}(K)}

\theoremstyle{plain}
\newtheorem{theorem}{Theorem}
\newtheorem{corollary}{Corollary}
\newtheorem{lemma}{Lemma}

\theoremstyle{definition}
\newtheorem{definition}{Definition}
\begin{document}
\title[Constant rank subspaces of symmetric forms]
{Dimension bounds for constant rank subspaces\\ 
of symmetric bilinear forms over a finite field}

\author[R. Gow]{Rod Gow}
\address{School of Mathematical Sciences\\
University College Dublin\\
 Ireland}
\email{rod.gow@ucd.ie}

\keywords{matrix, rank, constant rank subspace, symmetric bilinear form, finite field.}
\subjclass{15A03, 15A33}

\begin{abstract} 

Let $V$ be a vector space of dimension $n$ over the the finite field $\mathbb{F}_q$, where $q$ is odd, and 
let $\Symm(V)$ denote the space of symmetric bilinear forms defined on $V\times V$. 
We investigate constant rank $r$ subspaces of $\Symm(V)$ in this paper (equivalently, constant
rank $r$ subspaces of symmetric $n\times n$ matrices with entries in $\mathbb{F}_q$).
We have proved elsewhere that such a subspace has dimension at most $n$ if $q\geq r+1$ but in this paper we provide generally improved upper bounds.
Our investigations yield information about
common isotropic points of such subspaces, and also how the radicals of the elements in the subspace are distributed throughout $V$.

\end{abstract}
\maketitle
\section{Introduction}
\noindent  Let $K$ be a field and let $\matrices$  denote the space of $m\times n$ matrices over $K$, where we assume that $m\leq n$. We say that  a non-zero subspace $\M$ of $\matrices$ is a constant rank $r$ subspace if each non-zero element of
$\M$ has rank  $r$, where $1\leq r\leq m$.  

Let $q$ be a power of a prime, and let $\mathbb{F}_q$ denote the finite field of order $q$. We showed in a previous paper that a constant
rank $r$ subspace of $M_{m\times n}(\mathbb{F}_q)$ has dimension at most $n$ provided $q\geq r+1$, \cite{Gow}. This bound is optimal, subject
to the restriction on $q$, since for all $q$, $M_{m\times n}(\mathbb{F}_q)$ contains an $n$-dimensional constant rank $r$ subspace. 
For suitable values of $r$, $m$ and $n$, there are a few
constant rank $r$ subspaces of $M_{m\times n}(\mathbb{F}_2)$ of dimension greater than $n$, but we know of no other counterexamples when $q>2$.
(The absence of any other counterexamples may be partly explained by the difficulty of performing computer searches when $q>2$.)
We have the bound $\dim \M\leq m+n-r$ for any constant rank $r$ subspace of $M_{m\times n}(\mathbb{F}_q)$, valid for all $q$, but this bound is usually too large.

It follows that if $q\geq r+1$, a constant rank $r$ subspace of symmetric $n\times n$ matrices over $\mathbb{F}_q$ has dimension at most $n$, but there are reasons to expect that this upper bound can be improved when we exploit the symmetry of the matrices. This paper is devoted to finding
bounds for such constant rank subspaces of symmetric matrices. It builds on a previous paper, \cite{DGS}, and its key ingredient is an application of Lemma 1 of \cite{Fill}, which we exploited in \cite{Gow} to investigate constant rank subspaces.

For ease of exposition, we have stated our findings in terms of subspaces of symmetric bilinear forms defined on $V\times V$, where $V$ is a vector space of dimension $n$ over the field $K$ (and $K$ is usually $\mathbb{F}_q$). We have done this mainly to exploit the idea of the set
of common isotropic points of a subspace of symmetric bilinear forms, which arises naturally in the context of Theorem \ref{radical_totally_isotropic}, proved below. Our paper \cite{DGS} gave a formula for the number of common isotropic points in the finite field case, and we make good use of this formula
here. We are sure that the reader will have no trouble translating our results on constant rank subspaces of symmetric bilinear forms into identical results about constant rank subspaces of symmetric matrices. 

Throughout this paper, $\Symm(V)$ denotes the $K$-vector space of symmetric bilinear forms defined on $V\times V$.

\section{Common isotropic subspaces}

\noindent The following result is important for all our work in this paper. It is a straightforward application of the proof of Lemma 1
of \cite{Fill} to the context of symmetric matrices and symmetric bilinear forms.

\begin{theorem} \label{radical_totally_isotropic} Let $V$ be a vector space of dimension $n$ over the field $K$ and let $f$ and $g$ be non-zero elements of $\Symm(V)$. Suppose that $f$ has rank $r$ and that all $K$-linear combinations of $f$ and $g$ have rank
at most $r$. Then provided $|K|\geq r+1$, we have $g(u,w)=0$ for all elements $u$ and $w$ of the radical of $f$. Thus, the radical of $f$
is totally isotropic for all linear combinations of $f$ and $g$.

\end{theorem}

\begin{proof}
 
The result is trivial if $r=n$, so we may assume that $r<n$. We identify $V$ with $K^n$, and we choose a basis of $V$ with respect to which the
matrix of $f$ is
\[
C=\left(
\begin{array}
{cc}
        0&0\\
        0&A    
\end{array}
\right),
\] 
where $A$ is an invertible $r\times r$ symmetric matrix over $K$. The radical
of $f$ is then spanned by the $n-r$ standard basis vectors $e_1$, \dots, $e_{n-r}$ of $K^n$. 

Let the matrix of $g$ with respect to the same basis be
\[
 D=\left(
\begin{array}
{cc}
        A_1&A_2\\
        A_2^T&A_3    
\end{array}
\right),
\] 
where $A_1$ is an $(n-r)\times (n-r)$ symmetric matrix, $A_3$ is an $r\times r$ symmetric matrix, 
and $A_2$ is an $(n-r)\times r$ matrix.

For any element $x\in K$, the matrix of $g+xf$ with respect to the basis is
\[
 D+xC=
\left(
\begin{array}
{cc}
        A_1&A_2\\
        A_2^T&A_3+xA    
\end{array}
\right).
\]
We note that 
\[
 \det(A_3+xA)=\det A\det (A^{-1}A_3+xI_r)
\]
is a polynomial in $x$ of degree $r$ with leading coefficient $\pm \det A$.

Let $A_1=(a_{ij})$, $1\leq i,j\leq n-r$ and let $x_i$, $x_j$ denote the $i$-th and $j$-th rows of $A_2$ (we consider $x_i$ and $x_j$ as
row vectors of size $r$). Then
\[
 E=
\left(
\begin{array}
{cc}
        a_{ij}&x_i\\
        x_j^T&A_3+xA    
\end{array}
\right)
\]
is an $(r+1)\times (r+1)$ submatrix of $D+xC$, whose determinant must be $0$, since $E$ has rank at most $r$. As $\det E$ is a polynomial in
$x$ of degree at most $r$  whose coefficient of $x^r$ is $\pm a_{ij}\det A$, the supposition that $|K|\geq r+1$ implies that this polynomial is identically zero, and thus $a_{ij}=0$. This shows that $A_1$ is the zero matrix. It follows that $g$ is totally isotropic on the subspace spanned by
$e_1$, \dots, $e_{n-r}$, as required.

\end{proof}

\section{Spaces of symmetric bilinear forms of odd constant rank}

\noindent For most of rest of this paper, $V$ will denote a vector space
of dimension $n$ over $\mathbb{F}_q$. In \cite{DGS}, Theorem 2 and Corollary 3, we showed the following. Let $\M$ be a 
constant rank $r$ subspace of  $\Symm(V)$. Then if $r$ is odd, we have $\dim \M\leq r$. 

Furthermore,
$r$-dimensional constant rank $r$ subspaces of symmetric bilinear forms certainly exist. They may be constructed as follows. Let $U$ be a vector space of dimension $r$ over $\mathbb{F}_q$. Then there exists an $r$-dimensional constant rank $r$ subspace, $\N$, say, of symmetric bilinear
forms defined on $U\times U$. We can extend $\N$ to a subspace of  $\Symm(V)$ in an obvious way, where the extended forms all have the same radical of dimension $n-r$. 
Theorem \ref{common_radicals} below shows  that when $q$ and $r$ are odd, and $q\geq r+1$, this is the only way in which such constant rank  subspaces can be constructed.

We recall that if $\M$ is a subspace of  $\Symm(V)$, a vector $v\in V$ that satisfies $f(v,v)=0$
for all $f\in\M$ is called a \emph{common isotropic point} for the forms in $\M$.

\begin{theorem} \label{common_radicals}
Let $\M$ be a constant rank $r$ subspace of  $\Symm(V)$, where $V$ has
dimension $n$ over $\mathbb{F}_q$. Suppose that $q$ is odd and greater than $r$. Then if $r$ is odd, and $\dim \M=r$, all the non-zero
forms in $\M$ have the same radical. Thus, in this case, $\M$ is essentially defined on a space of dimension $r$.
\end{theorem}

\begin{proof} Let $f$ be a non-zero element of $\M$ and let $R$ denote the radical of $f$. Then $\dim R=n-r$. Moreover, as $q>r$, Theorem \ref{radical_totally_isotropic} implies that $R$ is totally isotropic for all the forms in $\M$. Thus, there are at least
$q^{n-r}$ common isotropic points for the elements of $\M$. Now, since all non-zero elements of $\M$ have odd rank, $q$ is odd, and
$\dim \M=r$, the total 
number of common isotropic points for $\M$ is $q^{n-r}$, by Theorem 5 of \cite{DGS}. It follows that $R$ is precisely the set of common isotropic points. However,
this applies to the radical of every non-zero element of $\M$. Thus, all radicals are identical, and $\M$ is essentially defined on
$V/R \times V/R$, where $V/R$ is $r$-dimensional.

\end{proof}

We can be more precise about  the radicals of the bilinear forms in a constant rank space if we examine the common isotropic points
more carefully. We recall that a \emph{subspace partition} of a finite-dimensional vector space $X$ over $\mathbb{F}_q$ is a collection of non-zero subspaces of $X$ such that each non-zero element of $X$ is in exactly one subspace of the partition. Thus, we may write
\[
 X=X_1 \cup \cdots \cup X_t,
\]
where $X_i\cap X_j=0$ if $i\neq j$. We say that the partition is non-trivial if $t>1$.

Subspace partitions have been studied in considerable detail because of their applications in finite geometry. We require an estimate for the smallest value of $t$ arising in a subspace partition. Our estimate must be common knowledge, and more refined estimates are available,
but we provide a proof of what we need, as it is not very complicated.

\begin{lemma} \label{partition_number}
Let $X$ be a vector space of dimension $n\geq 2$ over  $\mathbb{F}_q$ and let
\[
 X=X_1 \cup \cdots \cup X_t,
\]
where $\dim X_1\geq \dim X_2\geq \ldots \geq \dim X_t$,
describe a non-trivial subspace partition of $X$. 

Then, if $n=2m$ is even, $t\geq q^m+1$, and if $n=2m+1$ is odd, we have $t\geq q^{m+1}+1$.
\end{lemma}

\begin{proof}
 
Suppose first that $n=2m$. We begin by considering the case that $\dim X_1>m$. Then we have $\dim X_1=m+s$, where $1\leq s\leq m-1$. Since
$X_1\cap X_2=0$, we have
\[
 2m\geq \dim X_1+\dim X_2
\]
and thus $\dim X_2\leq m-s$. It follows from our labelling of indices that $\dim X_i\leq m-s$ for $i\geq 2$.

Suppose that there are exactly $a_i$ $i$-dimensional subspaces in the partition, where $1\leq i\leq m-s$. Then we have
\[
 q^{2m}-q^{m+s}=\sum_{i=1}^{m-s} a_i(q^i-1).
\]
Now 
\[
 \sum_{i=1}^{m-s} a_i(q^i-1)\leq (a_1+\cdots +a_{m-s})(q^{m-s}-1)=(t-1)(q^{m-s}-1).
\]
We deduce that
\[
 q^{m+s}(q^{m-s}-1)\leq (t-1)(q^{m-s}-1)
\]
and hence $t\geq q^{m+s}+1\geq q^{m+1}+1$. 

Next, we suppose that $\dim X_1\leq m$. The same argument as above then yields that
\[
 q^{2m}-1\leq t(q^m-1),
\]
and thus $t\geq q^m+1$. Thus in all cases we have therefore $t\geq q^m+1$.

Suppose now that $n=2m+1$ is odd. Again, we begin by considering the case that $\dim X_1=m+s+1$, where $0\leq s\leq m-1$. Then we must have
$\dim X_2\leq m-s$ and we obtain as before
\[
 q^{2m+1}-q^{m+s+1}\leq (t-1)(q^{m-s}-1).
\]
It follows that $t-1\geq q^{m+s+1}$ and hence $t\geq q^{m+s+1}+1\geq q^{m+1}+1$. 

Finally, suppose that $\dim X_1\leq m$. This time we obtain
\[
 q^{2m+1}-1\leq t(q^m-1)
\]
and thus
\[
 t\geq \frac{q^{2m+1}-1}{q^m-1}.
\]
Since the inequality
\[
 \frac{q^{2m+1}-1}{q^m-1}>q^{m+1}+1
\]
holds, we see once more that in all cases for odd $n$, $t\geq q^{m+1}+1$.

\end{proof}

The two bounds for $t$ are optimal, as is well known. When $n=2m$ is even, $X$ may be covered by a spread of $q^m+1$ subspaces of dimension
$m$. Suppose now that $n=2m+1$ is odd. Let $Y$ be a vector space of dimension $2m+2$ over $\mathbb{F}_q$ that is covered by 
a spread of $q^{m+1}+1$ subspaces of dimension $m+1$, say
\[
 Y=Y_1 \cup \cdots \cup Y_k,
\]
where $k=q^{m+1}+1$. We may choose $X$ to be a subspace of codimension 1 in $Y$ that contains $Y_1$. Then
\[
 X=(Y_1\cap X) \cup \cdots \cup (Y_k\cap X)=Y_1\cup \cdots \cup (Y_k\cap X)
\]
is a partition of $X$ into one subspace $Y_1$ of dimension $m+1$ and $q^{m+1}$ subspaces 
$Y_2\cap X$, \dots,  $Y_k\cap X$ of dimension $m$. Thus we have a partition of $X$ into $q^{m+1}+1$ subspaces. 

Let $\langle u\rangle$ denote the one-dimenional subspace of $V$ spanned by a non-zero vector $u$ in $V$ and let $\M$ denote a subspace
of $\Symm(V)$. We set
\[
 \M_{\langle u\rangle}=\{ f\in \M: u\in \rad f\}.
\]
It is clear that $\M_{\langle u\rangle}$ is a subspace of $\M$. 

Suppose now that $\M$ is a constant rank $n-1$ subspace. Let $\langle u\rangle$ and $\langle w\rangle$ be different one-dimensional subspaces of $V$ for which $\M_{\langle u\rangle}$ and $\M_{\langle w\rangle}$ are both non-zero. Then we claim that $\M_{\langle u\rangle}\cap \M_{\langle w\rangle}=0$. For, if $f$ is a non-zero element of $\M_{\langle u\rangle}\cap \M_{\langle w\rangle}$, the two-dimensional subspace spanned by
$u$ and $w$ is in $\rad f$, which contradicts the fact that $f$ has rank $n-1$. 

Continuing with the hypothesis that $\M$ is a constant rank $n-1$ subspace, let $\langle u_1\rangle$, \dots, $\langle u_t\rangle$ be all the different one-dimensional subspaces of $V$ such that for $1\leq i\leq t$, $\langle u_i\rangle$ is the radical of a non-zero element of $\M$. We see then that the $\M_{\langle u_i\rangle}$, $1\leq i\leq t$, form a subspace partition of $\M$. Furthermore, Theorem \ref{radical_totally_isotropic}
implies that $g(u_i,u_i)=0$ for all $g\in \M$ and $1\leq i\leq t$. Thus $\M$ has at least $t(q-1)$ non-zero common isotropic points. This fact will enable us to obtain more information about common radicals.

\begin{theorem} \label{improved_common_radicals}
 
 Suppose that $n$ is even and $\M$ is a $d$-dimensional constant rank $n-1$ subspace of $\Symm(V)$. Suppose also that $q\geq n$. Then all the non-zero elements of $\M$ have the same radical under any of the following conditions:

\begin{enumerate}
\item $n=6m$ and $4m\leq d\leq 6m-1$;
\item $n=6m+2$ and $4m+1\leq d\leq 6m+1$;
\item $n=6m+4$ and $4m+3\leq d\leq 6m+3$.
\end{enumerate}
\end{theorem}

\begin{proof}
 Let $\langle u_1\rangle$, \dots, $\langle u_t\rangle$ be all the different one-dimensional subspaces of $V$ that occur as the radical of a non-zero element of $\M$. Clearly, all the non-zero elements of $\M$ have the same radical precisely when $t=1$. Thus, we want to show that
$t=1$ under any of the stated conditions. We note that it suffices to prove the theorem whenever $d$ assumes the lower bound in each case.
We also note that $\M$ has $q^{n-d}-1$ non-zero common isotropic points by Theorem 5 of \cite{DGS}.

We consider the three possibilities in turn. Suppose that $n=6m$ and $d=4m$. Then $\M$ has exactly $q^{2m}-1$ non-zero common isotropic points. Suppose
if possible that $t>1$. Then, since we have a non-trivial partition of $\M$ into
$t$ subspaces, Lemma \ref{partition_number} implies that $t\geq q^{2m}+1$. Furthermore, we must have $t(q-1)\leq q^{2m}-1$. However,
\[
 t(q-1)\geq (q^{2m}+1)(q-1)>q^{2m}-1,
\]
and we have a contradiction. Thus, $t=1$ here.

Suppose next that $n=6m+2$ and $d=4m+1$. Then $\M$ has exactly $q^{2m+1}-1$ non-zero common isotropic points. Suppose
if possible that $t>1$. Then
Lemma \ref{partition_number} implies that $t\geq q^{2m+1}+1$ and  we must also have $t(q-1)\leq q^{2m+1}-1$. However,
\[
 t(q-1)\geq (q^{2m+1}+1)(q-1)>q^{2m+1}-1,
\]
and we have a contradiction. Thus, $t=1$ holds here as well.

Suppose finally that $n=6m+4$ and $d=4m+3$. Accordingly, $\M$ has exactly $q^{2m+1}-1$ non-zero common isotropic points. Suppose
that $t>1$. Then
Lemma \ref{partition_number} implies that $t\geq q^{2m+2}+1$ and  we must also have $t(q-1)\leq q^{2m+1}-1$. However,
\[
 t(q-1)\geq (q^{2m+2}+1)(q-1)>q^{2m+1}-1,
\]
and we have another contradiction. Thus, $t=1$ holds in all cases.

\end{proof}

\section{Constant rank $n-1$ subspaces when $n$ is odd}

\noindent Dimension bounds and analysis of structure for even constant rank subspaces are complicated by the fact that the formula for the number of common isotropic points is not so easy to use effectively, compared with the corresponding formula in the odd rank case. The difficulty arises from the fact that there are two different types of symmetric bilinear forms of even rank, namely, those of positive type and those of negative type.

We can define the type in various ways, and our nomenclature may not be universally used, but the two types are distinguished by the number of isotropic points for a given form. A symmetric bilinear form of even rank and positive type has more isotropic points than a symmetric bilinear form of the same rank
and negative type. See, for example, Lemma 1 of \cite{DGS}.

At present, we do not have any good idea of the number of elements of positive type compared with the number of elements of negative type
in an even constant rank space. This makes it more difficult to obtain good bounds for the dimension of an even constant rank
space  of symmetric bilinear forms. We confine ourselves to examining one special case in this section.

\begin{theorem} \label{even_constant_rank_dimension}
 
Let $n\geq 3$ be an odd integer and let $\M$ be a constant rank $n-1$ subspace of  $\Symm(V)$. Then if $q$ is odd and at least $n$, we have $\dim \M\leq n-1$. 
\end{theorem}

\begin{proof}
 We set $n=2m+1$ and follow the proof of Theorem \ref{improved_common_radicals} fairly closely.
Suppose if possible that $\dim \M=n$. Let $A$ be the number of non-zero elements in $\M$ of positive type and let 
$B$ be the number of non-zero elements in $\M$ of negative type. Then of course $A+B=q^n-1$ and the number of non-zero 
common isotropic points of $\M$ is
\[
 (A-B)q^{-m},
\]
by Theorem 5 of \cite{DGS}. 

Let $\langle u_1\rangle$, \dots, $\langle u_t\rangle$ be all the different one-dimensional subspaces of $V$ that occur as radicals of non-zero elements of $\M$. We claim that we cannot have $t=1$. For if $t=1$, all the non-zero elements of $\M$ have the same radical, 
$\langle u_1\rangle$. In this case, we set $\overline{V}=V/\langle u_1\rangle$. Then we can identify $\M$ with an $n$-dimensional constant rank
$n-1$ subspace of symmetric bilinear forms defined on $\overline{V}\times \overline{V}$. However, since $\dim \overline{V}=n-1$, the largest
dimension of a constant rank $n-1$ subspace of symmetric bilinear forms defined on $\overline{V}\times \overline{V}$ is $n-1$. This contradiction implies that $t>1$. 

We deduce that there is a non-trivial subspace partition of $\M$ by $t$ subspaces and since $\dim \M=2m+1$, we obtain the estimate
\[
 t\geq q^{m+1}+1
\]
from Lemma \ref{partition_number}. Now, as we are assuming that $q\geq n$, the radical of each non-zero element is totally isotropic
with respect to all elements of $\M$ and hence there are at least
\[
 (q^{m+1}+1)(q-1)
\]
non-zero common isotropic points of $\M$. It follows that
\[
 (q^{m+1}+1)(q-1)\leq (A-B)q^{-m}\leq (A+B)q^{-m}=(q^{2m+1}-1)q^{-m}
\]
and hence 
\[
 (q^{m+1}+1)(q^{m+1}-q^m)\leq q^{2m+1}-1.
\]
This latter inequality is impossible, and we have a contradiction. We deduce that $\dim \M\leq n-1$ when $q$ is odd and at least $n$.

\end{proof}

We note that the hypothesis that $q$ is odd is essential in Theorem \ref{even_constant_rank_dimension}. For, if $q$ is a power of 2, and $n=2m+1$ is odd, there is a $2m+1$-dimensional constant rank $2m$ subspace of alternating bilinear forms defined on $V\times V$. Since alternating bilinear forms are symmetric in characteristic 2, Theorem \ref{even_constant_rank_dimension} requires that $q$ is odd.

\section{Examples of constant rank $n-1$ subspaces}

\noindent Let $n\geq 3$ be an integer and suppose that $n+1$ is not a prime. Write $n+1=mr$, where $m$ and $r$ are integers, and $1<m<n+1$. Suppose that 
$K$ is an arbitrary field that has a separable field extension $L$ of degree $n+1$ and consider $L$ as a vector space over $K$. We
let $\Tr:L\to K$ denote the usual trace form.

 For each element $z$ of $L$, we define an element $f_z$, say, of $\Symm(L)$ by setting
\[
 f_z(x,y)=\Tr(z(xy))
\]
for all $x$ and $y$ in $L$. Each form $f_z$ is non-degenerate if $z\neq 0$, since the trace form is non-zero under the hypothesis of separability.

Suppose now that $L$ contains a subfield $M$ such that $K<M<L$ and $M$ has degree $m$ over $K$. (The hypothesis is obviously met
if $K= \mathbb{F}_q$.)  
Let $V$ be any $K$-subspace of codimension 1 in $L$ that contains $M$ and let $f_z'$ denote the restriction of $f_z$ to
$V\times V$ for each $z\in L$.
Let $\mathcal{M}$ denote the subspace of all bilinear forms
$f_z'$. 

Let us assume that $z\neq 0$. Then since $f_z$ is non-degenerate, $f'_z$ has rank $n$ or $n-1$. Now, 
$\mathcal{M}$ is an $(n+1)$-dimensional subspace of  $\Symm(V)$ and hence its non-zero elements cannot all have rank $n$. We deduce that there exists some $z\in L$ such that $f_z'$ has rank $n-1$. Let $\langle u\rangle$ be the radical of $f_z'$.  It is then straightforward
to see that the radical of $f'_{zu}\in \M$ is spanned by the unity element $1$. 

Consider now the set $\K$, say, of all symmetric bilinear forms $f'_{zuw}$, as $w$ runs over the subfield $M$. It is again easy to see that if $w\neq 0$,
$f'_{zuw}$ has rank $n-1$, and its radical is $\langle w^{-1}\rangle$. Thus the radicals are different subspaces provided we use elements
of $M$ that are not $K$-multiples of each other. Furthermore, $\K$ is a $K$-subspace of dimension $m$.

This discussion provides the justification for the following result.

\begin{theorem} \label{different_radicals}
 
Let $V$ be a vector space of dimension $n$ over $\mathbb{F}_q$, where $n\geq 3$. If $n=2m-1$ is odd, there exists an $m$-dimensional constant rank 
$n-1$ subspace of $\Symm(V)$, in which all linearly independent forms have different radicals. 

If $n$ is even and $n+1=3m$ is divisible by $3$, there also exists an $m$-dimensional constant rank 
$n-1$ subspace of $\Symm(V)$, in which all linearly independent forms have different radicals. 
\end{theorem}

The subspace described in the first part of Theorem \ref{different_radicals} has maximal dimension subject to the radicals being different
and $q$ sufficiently large, as we now show.

\begin{theorem} \label{maximum_dimension_different_radicals}
 Let $V$ be a vector space of dimension $n$ over $\mathbb{F}_q$, where $n=2m-1$ is odd and at least $3$. 
Let $\M$ be a $d$-dimensional constant rank 
$n-1$ subspace of $\Symm(V)$, in which all linearly independent forms have different radicals. Then if
$q\geq n$, we have $d\leq m$.

\end{theorem}

\begin{proof}
Suppose, if possible, that $d=m+1$. Then there are $(q^{m+1}-1)/(q-1)$ different one-dimensional subspaces of $V$ that each occur as
the radical of a non-zero element of $\M$. Since we are assuming that $\M$ is a constant rank space and $q\geq n$, Theorem \ref{radical_totally_isotropic} implies that we obtain thereby $q^{m+1}-1$ non-zero common isotropic points for $\M$. 

Let $A$ be the number of non-zero elements in $\M$ of positive type and let 
$B$ be the number of non-zero elements in $\M$ of negative type. Then the number of  
common isotropic points of $\M$ is
\[
q^{2m-1-(m+1)}+ (A-B)q^{2m-1-(m+1)-(m-1)}=q^{m-2}+(A-B)q^{-1},
\]
by Theorem 5 of \cite{DGS}. 

Comparing the estimate for the number of common isotropic points with the formula above, we deduce that
\[
 q^{m+1}\leq q^{m-2}+(A-B)q^{-1}\leq q^{m-2}+q^m.
\]
This inequality is clearly impossible, and hence $d\leq m$.

\end{proof}

\section{An unusual constant rank $4$ subspace}

\noindent Let $V$ denote the field $\mathbb{F}_{3^5}$ considered as a 5-dimensional vector space over $\mathbb{F}_{3}$. Let $\phi: V\times V\times V\to  \mathbb{F}_{3}$ be the symmetric trilinear form given by
\[
 \phi(x,y,z)=\Tr (x^9yz+xy^9z+xyz^9),
\]
where $\Tr$ is the trace form from $\mathbb{F}_{3^5}$ to $\mathbb{F}_{3}$. 

For $x\in \mathbb{F}_{3^5}$, let $\phi_x:V\times V\to \mathbb{F}_{3}$ be defined by
\[
 \phi_x(y,z)=\phi(x,y,z).
\]
It is clear that $\phi_x$ is an element of $\Symm(V)$ and the set of all such $\phi_x$, as $x$ ranges over $\mathbb{F}_{3^5}$, is a vector space, $\M$, say, over $\mathbb{F}_{3}$.  

It turns out that $\M$ is a 5-dimensional constant rank 4 subspace. We note that Theorem \ref{even_constant_rank_dimension} shows
that such a subspace does not exist over $\mathbb{F}_{q}$, when $q\geq 5$, and this fact shows that, in at least one non-trivial case,
the theorem is optimal. 

An expanation of why $\M$ has this constant rank property is by no means straightforward and it is embedded in a more general exploration
of the trilinear form $\phi$ carried out by Ward in a paper published in 1975, \cite{Ward}. He showed that $\M$ is  invariant under a linear action of the Mathieu group $M_{11}$, of order 7920. This means that $\M$ is also invariant under $M_{11}$, and this accounts for some of the special aspects of this constant rank space. 

We can give a short description of the simpler symmetries of $\phi$ and $\M$, but refer to Ward's paper for a complete exposition of this fascinating subject. Let $\epsilon$ denote an element of order 11 in $\mathbb{F}_{3^5}$ and let $\sigma$ denote the Frobenius automorphism
$x\to x^3$ of $\mathbb{F}_{3^5}$. It is easy to see that multiplication by $\epsilon$ defines a linear transformation $T$, say, of order 11 of $V$, which acts irreducibly on the vector space. It is clear that 
\[
 \phi(Tx,Ty,Tz)=\phi(x,y,z),
\]
which shows that $T$ fixes $\phi$ and hence leaves $\M$ invariant. 

It is also clear that $\sigma$ defines a linear transformation, $S$, say, of $V$  of order 5, which satisfies
$S^{-1}TS=T^4$. The invariance of the trace form with respect to Galois automorphisms implies that
\[
 \phi(Sx,Sy,Sz)=\phi(x,y,z)
\]
and thus $\M$ is  invariant under the action of $S$. Together, $S$ and $T$ generate a metacyclic group, $G$, say, of order 55, which fixes
$\phi$ and leaves $\M$ invariant.

$G$ acts irreducibly on $V$, since $T$ does. $V$ is the direct sum of $\langle 1\rangle$ and the trace zero hyperplane that are both
$S$-invariant. These are the only non-trivial subspaces of $V$ that are $S$-invariant. Now the set of all elements $x$ such that $\phi_x=0$ is a
$G$-invariant subspace of $V$. Since $G$ acts irreducibly on $V$, this subspace is trivial and hence $\phi_x=0$ only when $x=0$. This explains
why $\dim M=5$.

We may easily verify that $S$ fixes $\phi_1$ and $\langle 1\rangle$ is contained in $\rad \phi_1$. Now $\rad \phi_1$ is invariant under $S$, and since the only proper subspaces of $V$ that are $S$-invariant are $\langle 1\rangle$ and the trace zero hyperplane, as we remarked above, it follows that $\rad \phi_1=\langle 1\rangle$ and thus $\phi_1$ has rank 4. The orbit of $\phi_1$ under $G$ consists of the eleven
elements $\phi_{\epsilon^j}$, $1\leq j\leq 11$. Since a symmetric bilinear form of rank 4 and positive type is not invariant
under an element of order 5, $\phi_1$ has negative type. Thus $\M$ has at least 22 elements of rank 4 and negative type: the eleven forms
$\phi_{\epsilon^j}$ and their negatives.

In its action on $V\setminus 0$, $G$ has six orbits, two orbits being of size 11, the others of size 55. The same holds for the action of
$G$ on $\M$. Furthermore, since $G$ has odd order, none of its elements can map a non-zero element of $V$ or $\M$ into its negative and 
we deduce that that each $G$-orbit is paired with the orbit consisting of the negative of each element in the orbit. 

The 22 elements $\pm \epsilon^j$, $1\leq j\leq 11$ are non-zero common isotropic points of $\M$. Since the common isotropic points are invariant under $G$, if there are more than 22 of them, there are at least $22+110=132$ of them. This number, however, is greater than the number of non-zero  
isotropic points of
$\phi_1$ alone, which is 62. Thus there are precisely 22 non-zero common isotropic points for $\M$.

From our formula for the number of common isotropic points of a subspace of symmetric bilinear forms (Theorem 5 of \cite{DGS}), we can then deduce
that $\M$ has two possible structures. Either it contains 22 elements of rank 4 and negative type, and 220 elements of rank 4 and positive type, or else it contains 132 elements of rank 4 and negative type, and 110 elements of rank 2 and positive type. It seems that to decide which of these two possibilities actually occurs requires more detail than we have provided above and we must defer to Ward's paper for resolution of this issue.

We would like to point out further noteworthy features of the subspace $\M$. Let $v$ be a non-zero vector different from the 22 common isotropic points of $\M$. We map $\M$ linearly into $\mathbb{F}_{3}$ by associating $f\in\M$ with $f(v,v)$. There is thus a 4-dimensional
subspace $\N$, say, of $\M$, consisting of those elements $f$ that satisfy $f(v,v)=0$.

Let $A$ denote the number of non-zero elements of positive type in $\N$, and $B$ denote the number of non-zero elements of negative type in $\N$.
Now $\N$ has at least 24 non-zero common isotropic points, namely the 22 of $\M$, plus $\pm v$. Theorem 5 of \cite{DGS} implies that
\[
 25\leq 3+(A-B)3^{-1},
\]
and also that 3 divides $A-B$. 

We see that $A-B\geq 66$ and since $A+B=80$, we obtain that $A\geq 73$. Now $A$ is even, since if $f$ is an element of positive type in
a subspace, so also is $-f$. Thus $A\geq 74$. However, we cannot have $A=74$, $B=6$, since 3 does not divide $A-B$ in this case. Likewise,
$A=78$, $B=2$ is impossible. It follows that $A=76$, $B=4$, and $\N$ has exactly 26 non-zero common isotropic points, the 22 of $\M$, $\pm v$, and
$\pm w$, where $w$ is some other vector.

On the other hand, if we take an arbitrary 4-dimensional subspace of $\M$, it is easy to see that it either contains exactly 23 common isotropic
points, and hence 70 elements of positive type, and 10 of negative type, or it contains exactly 27 common isotropic points, and hence is of the type described above. By counting, we find that there are 66 subspaces of the former type, and 55 of the latter type (note that there are
$121=66+55$ 4-dimensional subspaces in $\M$). 

It seems probable that the two types of 4-dimensional subspaces of $\M$ are single orbits  under the action of $M_{11}$.

Finally, we mention that the author and his collaborators gave an example of a 5-dimensional constant rank 4 subspace of $5\times 5$ symmetric matrices with entries in $\mathbb{F}_{3}$, \cite{DGS}, Example 1. This subspace contains 220 elements of rank 4 and positive type, and 22 of rank 4 and negative type. At the time of writing of \cite{DGS}, we were unaware of Ward's paper and its relevance to the constant rank theme. The example was found by a computer search, and it is difficult to verify by hand that the five given matrices do indeed span a 5-dimensional constant rank space.

\section{Constant rank $2$ subspaces}

\noindent In this section, we initially let $K$ be an arbitrary field of characteristic different from 2. Let $V$ be a vector space of finite dimension $n$
over $K$. We would like to investigate constant rank 2 spaces of $\Symm(V)$. We begin by making a definition.

\begin{definition} \label{hyperbolic}
Suppose that $n\ge 2$. 
Let $f$ be an element of $\Symm(V)$ of rank 2. We say that $f$ is of \emph{hyperbolic type} if there is a vector $w$ not in the
radical of $f$ that satisfies $f(w,w)=0$.

\end{definition}

We note that when $K$ is a finite field of odd characteristic, a symmetric bilinear form of rank 2 and hyperbolic type is the same as one
of positive type.

\begin{lemma} \label{same_radical}
Let $K$ be a field of characteristic different from $2$ and let $V$ be a finite dimensional vector space over $K$, with $\dim V\geq 2$.
Let $f$ and $g$ be elements of $\Symm(V)$ of rank $2$, and suppose that $g$ is not of hyperbolic type. Suppose also that all non-trivial
$K$-linear combinations of $f$ and $g$ have rank $2$. Then $f$ and $g$ have the same radical, of dimension $n-2$. 

\end{lemma}

\begin{proof}

 We note that as $K$ has odd characteristic, we have $|K|\geq 3$, and hence Theorem \ref{radical_totally_isotropic} applies. Thus, let $R$ be the radical of $f$.
Theorem \ref{radical_totally_isotropic} implies that $g$ is totally isotropic on $R$. It follows that if $R$ is not the radical of $g$, $g$ is of hyperbolic type, a contradiction. Thus $f$ and $g$ have the same radical. 

\end{proof}

\begin{corollary} \label{dimension_at_most_2}
 Let $K$ be a field of characteristic different from $2$ and let $V$ be a finite dimensional vector space over $K$, with $\dim V\geq 2$.
Let $\M$ be a constant rank $2$ subspace of  $\Symm(V)$ that contains an element that is not of hyperbolic type. 
Then $\dim \M\leq 2$. 

\end{corollary}

\begin{proof}
 
Let $g$ be any element of $\M$ that is not of hyperbolic type and let $R$ be the radical of $g$. Let $f$ be any other non-zero element of
$\M$. By Lemma \ref{same_radical}, $R$ is also the radical of $f$. Thus $R$ is the common radical of all non-zero elements
of $\M$.

Let $\overline{V}=V/R$. It is straightforward to see that $\M$ is isomorphic to a constant rank 2 subspace of $\Symm(\overline{V})$, since
$R$ is the radical of each non-zero element of $\M$. Since $\dim \overline{V}=2$, it follows that $\dim \M\leq 2$.
\end{proof}

\begin{theorem} \label{dimension_at_most_n-1}
 Let $V$ be a  vector space of dimension $n\geq 2$ over $\mathbb{F}_{q}$, where $q$ is a power of an odd prime, 
and 
let $\M$ be a constant rank $2$ subspace of $\Symm(V)$. Suppose that each non-zero element
of $\M$ is of hyperbolic type. Then $\dim \M\le n-1$. 

\end{theorem}

\begin{proof}
Suppose, if possible, that $\dim M\geq n$. Then, taking if necessary a subspace of dimension $n$ in $\M$, we may assume that $\dim M=n$.
We know then that in this case the number of common isotropic points is
\[
 q^{n-n}-1+(q^n-1)q^{n-n-1}=q^{n-1}-q^{-1}.
\]
However, since this number is not an integer, we have a contradiction. Thus, $d\leq n-1$, as required.

\end{proof}

Theorem \ref{dimension_at_most_n-1} is optimal, since $(n-1)$-dimensional constant rank 2
subspaces of $\Symm(V)$ can be constructed, in which all non-zero elements have hyperbolic type. See, for example, Theorem 7 of \cite{DGS}. We also note that Theorem \ref{dimension_at_most_n-1} does not hold
necessarily for all fields of characteristic different from $2$. For example, consider the real $2\times 2$ matrices
\[
\left(
\begin{array}
{rr}
        a&b\\
        b&-a    
\end{array}
\right),
\]
where $a$ and $b$ run over the real numbers. It is easy to verify that if $a$ and $b$ are not both zero, the matrix defines a real symmetric bilinear
form of rank 2 and hyperbolic type. Thus we have a two-dimensional real subspace of symmetric bilinear forms of rank 2 and hyperbolic type, defined
on a two-dimensional real vector space.

\section{Constant rank $4$ subspaces}

\noindent We would like to extend the analysis of the previous section to investigate general subspaces of symmetric bilinear forms of even constant rank. To this end, we return to the hypothesis that $V$ has dimension $n$ over $\mathbb{F}_{q}$.

We wish to draw attention to the following relevant facts. Suppose that $t$ is a positive integer with $2t\leq n$. Then the largest dimension
of a constant rank $2t$ subspace of $\Symm(V)$, in which each non-zero element has positive type, is $n-t$. Furthermore, constant rank $2t$ subspaces of $\Symm(V)$ of dimension $n-t$ and positive type exist for all $q$. See, for example, Theorems 6 and 7 of \cite{DGS}.

We suspect that if $n>3t$, the largest dimension of a constant rank $2t$ subspace of $\Symm(V)$ is $n-t$ and if this dimension is realized, then all non-zero elements in the subspace have positive type. We may need some requirement that $q\geq 2t+1$ for this to be true. Theorem \ref{dimension_at_most_n-1} has confirmed this bound when $t=1$, the simplest case.

At present, proof of this proposed upper bound for the dimension is missing, even in the probably simplest remaining case, when $t=2$. We will nonetheless present a proof of a weaker statement, namely, if $q$ is odd and at least 5, and $\M$ is a constant rank 4 subspace of $\Symm(V)$, then
$\dim \M\leq n-1$ if $n\geq 5$. We would expect the bound $\dim \M\leq n-2$ to hold, but we have been unable to achieve this. As will be seen, our proof involves some careful analysis of the way that the radicals of the non-zero elements of $\M$ intersect, and may not be the best
approach to the more general dimension problem.

Let $\M$ be a subspace of $\Symm(V)$ and let $U$ be a subspace of $V$. Extending the notation of  Section 3, we set
\[
 \M_U=\{ f\in \M: U\leq \rad f\}.
\]
It is easy to verify that $\M_U$ is a subspace of $\M$. 

\begin{lemma} \label{radical_intersection}
 
Let $\M$ be a constant rank $4$ subspace of $\Symm(V)$ of dimension at least $2$ and suppose that $q\geq 5$. Suppose that $\M$ contains an element,
$f$, say, of negative type and let $R=\rad f$. Let $g$ be any other non-zero element of $\M$ and let $S=\rad g$. Then $R\cap S$ has codimension at most $1$ in both $R$ and $S$. 
\end{lemma}

\begin{proof}
 
As $q\geq 5$, Theorem \ref{radical_totally_isotropic} implies that $S$ is totally isotropic for $f$. Let $\overline{f}$ be the element of
$\Symm(V/R)$ induced by $f$. Then $\overline{f}$ has rank 4 and negative type. We deduce that any non-zero subspace of $V/R$ that is totally isotropic for
$\overline{f}$ is one-dimensional. Now
\[
 (R+S)/R\cong S/R\cap S
\]
is totally isotropic for $\overline{f}$, by our remarks above. Hence $\dim (S/R\cap S)\leq 1$. It follows that $R\cap S$ has codimension at most 1 in $S$, and hence also in $R$, since $\dim R=\dim S$.
\end{proof}

We continue with the hypothesis of Lemma \ref{radical_intersection}, but assume additionally that $n\geq 5$. We have $\dim R=n-4$. Let $U_1$,
\dots, $U_t$ be all the subspaces of $R$ of codimension 1 in $R$, where
\[
 t=\frac{q^{n-4}-1}{q-1}.
\]
We set
\[
 \M_i=\M_{U_i}
\]
for $1\leq i\leq t$. We note that $\M_R\leq M_i$ for all $i$, since $U_i\leq R$. 

\begin{lemma} \label{M_i_intersections}
Adhering to the notation introduced above, we have $\M_i\cap M_j=\M_R$ if $i\neq j$. We also have $\dim \M_i\leq 4$ for all $i$ and
$\dim \M_R\leq 4$. 

\end{lemma}

\begin{proof}
 
Suppose that $i\neq j$ and let $h$ be an element of $\M_i\cap \M_j$. Then $U_i\leq \rad h$ and $U_j\leq \rad h$, and hence $U_i+U_j\leq \rad h$.
Since $U_i\neq U_j$, and each subspace has codimension 1 in $R$, we obtain that $R\leq \rad h$. This implies that $h\in \M_R$. On the other hand,
since $\M_R\leq \M_i$ and $\M_R\leq \M_j$, it is clear that $\M_R\leq \M_i\cap M_j$ and we therefore have the equality  $\M_i\cap M_j=\M_R$.

Suppose that $\theta$ is an element of $\M_i$. Then $U_i\leq \rad \,\theta$ and thus $\theta$ determines an element $\overline{\theta}$, say,
of $\Symm(V/U_i)$, where $\overline{\theta}$ has rank 4. Let $\overline{\M_i}$ denote the image of $\M_i$ under the mapping sending
$\theta$ to $\overline{\theta}$. Since $\theta\to \overline{\theta}$ is linear, and all elements of $\M_i$ vanish on $U_i$, $\overline{\M_i}$ is a subspace of $\Symm(V/U_i)$ and $\dim \overline{\M_i}=\M_i$. 

Now $\dim (V/U_i)=5$ and $\overline{\M_i}$ is a constant rank 4 subspace of $\Symm(V/U_i)$. Since we are assuming that $q\geq 5$, $\dim \overline{\M_i}\leq 4$ by Theorem \ref{dimension_at_most_n-1}. Thus $\dim \M_i\leq 4$ for all $i$. Finally, since $\M_R\leq M_i$, $\dim \M_R\leq 4$. However, we can equally well regard $\M_R$ as a constant rank 4 subspace of $\Symm(V/R)$. Since $\dim (V/R)=4$, it follows that
$\dim \M_R\leq 4$ (and this inequality holds for all $q$).

\end{proof}

\begin{theorem} \label{constant_rank_4_dimension}
 
Suppose that $q$ is odd and at least $5$. Let $\M$ be a constant rank $4$ subspace of $\Symm(V)$. Then if $n=\dim V\geq 5$, we have $\dim \M\leq n-1$.
\end{theorem}

 \begin{proof}
If $n=5$, the result follows from Theorem \ref{even_constant_rank_dimension}. Thus we may assume that $n\geq 6$. 

Suppose now that all non-zero elements of $\M$ have positive type. Then $\dim \M\leq n-2$ by Theorem 6 of \cite{DGS}, and we are finished in this case.

We may therefore assume that $\M$ contains a non-zero element, $f$, say, of negative type, and we set $R=\rad f$. In accordance with our previous discussion, we define the subspaces $\M_1$, \dots, $\M_t$ of $\M$, where $t=(q^{n-4}-1)/(q-1)$. Lemma \ref{radical_intersection} implies that
each element of $\M$ is contained in some subspace $\M_i$.

Consider now the vector space $\M/\M_R$. If $i\neq j$, the subspaces $\M_i/\M_R$ and $\M_i/\M_R$ intersect trivially, by Lemma \ref{M_i_intersections}. Thus $\M/\M_R$ is the  union of the subspaces $\M_i/\M_R$, but this may not be a subspace partition as some of the subspaces may be the zero subspace. 

Suppose if possible that $\dim \M=n$. We set $\dim \M_R=e$, where $1\leq e\leq 4$, since $f\in \M_R$. Then $\M/\M_R$ is the union of at most
$t$ subspaces each of dimension at most $4-e$, intersecting trivially. Since $\dim (\M/\M_R)=n-e$, we obtain
\[
 q^{n-e}-1\leq t(q^{4-e}-1)=\frac{(q^{n-4}-1)(q^{4-e}-1)}{q-1}.
\]
This yields the inequality
\[
 (q^{n-e}-1)(q-1)\leq (q^{n-4}-1)(q^{4-e}-1),
\]
which in turn yields
\[
 q^{n-e+1}\leq 2q^{n-e}-q^{n-4}-q^{4-e}+q,
\]
which is clearly impossible. Thus we have $\dim \M\leq n-1$, as required.

 \end{proof}

\section{Conclusions}

\noindent We have shown that in some cases, the dimension of a constant rank subspace of $\Symm(V)$ is less than $n$. Furthermore, we have shown, when the rank is odd, that the radicals of the non-zero elements of a constant rank space of sufficiently large dimension are all the same. On the other hand, we have provided examples that show that the radicals in certain constant rank subspaces of reasonably large dimension can all be different.

It is reasonable to assert that improved dimension bounds for constant rank subspaces of $\Symm(V)$ depend on answers to two questions.
One answer needed is what are the relative numbers of elements of positive type and of negative type in a constant rank subspace.
Another desideratum is a better understanding of how radicals are distributed throughout $V$. At two extremes, all radicals are the same or all
radicals are different.

Finally, it would be interesting to find further counterexamples to our dimension bounds when $q$ is small, and to look for sporadic constant rank subspaces with as rich a structure as that displayed in Section 6.

\end{document}